\documentclass{amsart}

\usepackage{latexsym}
\usepackage{amsmath,amsthm}
\usepackage{amssymb}
\usepackage{graphicx}
\usepackage{hyperref}

\input diagxy
 \xyoption{curve}

\newtheorem{theorem}{Theorem}[]
\newtheorem{lemma} [theorem]{Lemma}
\newtheorem{proposition} [theorem]{Proposition}
\newtheorem{definition} [theorem]{Definition}

\newtheorem{corollary} [theorem]{Corollary}
\newtheorem{remark}{Remark}
\newtheorem{example}{Example}

\begin{document}

\title{On $\mathcal{B}$-Open Sets}
\author{Layth M. Alabdulsada }

\address{Institute of Mathematics, University of Debrecen, H-4002 Debrecen, P.O. Box 400, Hungary}
\email{layth.muhsin@science.unideb.hu}

\subjclass[2000]{54C05, 54C08, 54C10} \keywords{Operator topological space, bi-operator topological space,
$\mathcal{B}$-open sets, $T^*$-open sets, contra-$\mathcal{B}$-continuous, Urysohn space, weakly Hausdorff space}

\bibliographystyle{alpha}

\begin{abstract}
The aim of this paper is to define and study $\mathcal{B}$-open sets and related properties.
A $\mathcal{B}$-open set is, roughly speaking, a generalization of a $b$-open set, which is in turn a generalization of a pre-open set and a semi-open set. Using $\mathcal{B}$-open sets, we introduce a number of concepts such as $\mathcal{B}$-dense, $\mathcal{B}$-Frechet, contra-$\mathcal{B}$-closed graph and contra-$\mathcal{B}$-continuity. Also, we define a bi-operator topological space $(X, \tau, T_1, T_2)$ which involves two operators $T_1$ and $T_2$, which are used to define $\mathcal{B}$-open sets.
 \end{abstract}
\maketitle

\section{Introduction}
Over the past years, an amount of generalizations of open sets has been considered.
The first notion due to Levine \cite{L1} in 1963 was semi-open sets, while in 1965 Nj{\aa}stad \cite{N2} introduced some classes of nearly open sets, more precisely, they investigated the structure of $\alpha$-open set and gave some applications. Mashhour et al. in 1982 \cite{M1} introduced and studied pre-open sets and pre-continuous functions.
In 1983 Abd El-Monsef et al. \cite{A1} introduced the new topological notions, $\beta$-open sets, $\beta$-continuous mappings and $\beta$-open mappings.
In 1996 \cite{A4}, Andrijevi\'{c} introduced and studied a new class of generalized open sets in a topological space, called $b$-open sets.
All of these above concepts were defined similarly using the closure operator Cl and the interior operator Int.

This research area (which is fertile in information) still takes a significant part of the investigations because it has a clear effect on the development of the topological space through the experience of many theories and characteristics of different types of open sets, for instance see (\cite{A2}, \cite{B1}, \cite{E1}, \cite{E3}, \cite{E4}, \cite{E5} and \cite{N3}).

We work on circulating the $b$-openness from a different point of view than previously stated since our generalization depends entirely on operators attached with topology $\tau$ on $X$ to define the $\mathcal{B}$-open sets. More accurately, let $P(X)$ be the power set of $X$ and functions $T_1, T_2: P(X) \to P(X)$ are operators associated with topology $\tau$ on $X$. Then the quadruple $(X, \tau, T_1, T_2)$ is called a bi-operator topological space.
However, if $T_1(S)= \mathrm{Cl(Int}(S))$ and $T_2(S)= \mathrm{Int(Cl}(S))$, then the notion of $\mathcal{B}$-open sets became exactly the same as the definition of the $b$-open sets. The use of the operator topological spaces for the first time goes back to H. J. Mustafa et al. \cite{M21}, \cite{M3}, and recently Alabdulsada \cite{LA}.

In this paper, first we introduce and study the new notion of bi-operator topological spaces and its related properties. Our generalization of open sets in topological space is called $\mathcal{B}$-open sets, which linked to bi-operator topological spaces. First we recall several concepts and definitions that contributed to constructing our definition, namely $\mathcal{B}$-open which generalizes $b$-open sets in a topological space. Afterwards, we apply $\mathcal{B}$-open sets to define some further new concepts, and show some remarks and examples for $\mathcal{B}$-open sets.

Our main results are given in Section 3, where we present and study several different spaces as well as functions which are based on $\mathcal{B}$-open sets. Also, we investigate the relationships between these types of functions, besides we check the relationships with some special spaces such as Urysohn space or  weakly Hausdorff space. To be precise, we prove that, among others, if the function $f: (X, \tau, T_1, T_2) \to (Y, \sigma)$ has a contra-$\mathcal{B}$-closed graph, then the inverse image of a contra-compact set $S$ of $Y$ is $\mathcal{B}$-closed in $X$. In addition, if $f : (X, \tau, T_1, T_2) \to (Y, \sigma)$ is contra-$\mathcal{B}$-continuous from a $\mathcal{B}$-connected space onto $Y$, then $Y$ is not a discrete space. Another new result says if $f : (X, \tau, T_1, T_2) \to (Y, \sigma)$ is a contra-$\mathcal{B}$-continuous surjective function and $X$ is $\mathcal{B}$-compact, then $Y$ is contra-compact. Furthermore, a number of important related properties are stated and proved.

\section{Background}

In this section, we recall and introduce some of the definitions and the fundamental notions that play a key role in this paper. Throughout,  $(X, \tau), \ (Y, \sigma)$ are arbitrary topological spaces and $S \subseteq X$. The closure of $S$ will be denoted by $\mathrm{Cl}(S)$. The interior of $S$ will be denoted by $\mathrm{Int}(S)$.

\begin{definition} \label{11}
A subset $S$ of a topological space $(X, \tau)$ is said to be:
  \begin{itemize}
    \item [(i)] \textsl{regular open}, if $S = \mathrm{Int(Cl}(S))$, \textsl{regular closed} if $S = \mathrm{Cl(Int}(S))$ \cite{ST}.
    \item [(ii)] \textsl{pre-open}, if $S \subseteq \mathrm{Int(Cl}(S))$, the complement of a pre-open is \textsl{pre-closed} \cite{M1}.
    \item [(iii)] \textsl{semi-open}, if  $S \subseteq \mathrm{Cl(Int}(S))$, the complement of a semi-open is \textsl{semi-closed}  \cite{L1}.
    \item [(iv)] \textsl{$\alpha$-open}, if $S \subseteq \mathrm{Int( Cl(Int}(S)))$, the complement of an $\alpha$-open is \textsl{$\alpha$-closed} \cite{N2}.
    \item [(v)] \textsl{$\beta$-open}, if $S \subseteq \mathrm{Cl(Int(Cl}(S)))$, the complement of a $\beta$-open is \textsl{$\beta$-closed} \cite{A1}.
    \item [(vi)] \textsl{$b$-open}, if $S \subseteq \mathrm{Cl(Int}(S)) \cup \mathrm{Int( Cl}(S))$, the complement of a $b$-open is \textsl{$b$-closed} \cite{A4}.
  \end{itemize}
\end{definition}
 In particular, the $\beta$-closure of a set $S$ denoted by $\beta\mathrm{Cl}(S)$,
is the intersection of all $\beta$-closed sets containing $S$. The $\beta$-interior of a set $S$ denoted by $\beta\mathrm{Int}(S)$, is the union of all $\beta$-open sets contained in $S$. The preclosure, preinterior, semiclosure, semiinterior, $b$-closure and $b$-interior of a set $S$ denoted by $\mathrm{pCl}(S), \mathrm{pInt}(S), \mathrm{sCl}(S),  \mathrm{sInt}(S),$ $b\mathrm{Cl}(S) \ \text{and} \ b\mathrm{Int}(S) $, respectively, are defined analogously.
\begin{proposition} \label{1} \cite{A4}
  Let S be a subset of a space X. Then:
  \begin{itemize}
    \item [(i)] $\mathrm{pInt}(S) = S \cap \mathrm{Int(Cl}(S)),$
    \item[(ii)] $\mathrm{pCl}(S) = S \cup \mathrm{Cl(Int}(S)),$
    \item [(iii)] $\mathrm{sInt}(S) = S \cap \mathrm{Cl(Int}(S)),$
    \item [(iv)] $\mathrm{sCl}(S) = S \cup \mathrm{Int(Cl}(S)).$
  \end{itemize}
\end{proposition}

\begin{definition} \cite{M21}
 Let $(X, \tau)$ be a topological space  and $P(X)$ be the power set of $X$. A function $$T: P(X) \to P(X)$$ is said to be an operator associated with topology $\tau$ on $X$ if $U \subseteq T(U)$ for all $U \in \tau$ and the triple $(X, \tau, T)$ is called an {\em operator topological space}.
\end{definition}

\begin{definition}
  Let $(X, \tau, T)$ be an operator topological space and $S \subseteq X $, then
  \begin{itemize}
    \item [(i)] $S$ is said to be \textsl{$T$-open} \cite{M21}, if for each $x \in S$ there exists $U \in \tau$ such that  $x \in U \subseteq T(U) \subseteq S$.
  The complement of $T$-open is called {\em $T$-closed}.
     \item [(ii)] $S$ is said to be {\em $T^*$-open} \cite{M3}, if $S \subseteq T(S)$ (observe that $S$ not necessarily open).
 The complement of $T^*$-open is called {\em $T^*$-closed}.
      \end{itemize}

\end{definition}
 \begin{remark}
  $T_1\mathrm{Cl}(S), \ T_2\mathrm{Cl}(S)$ are the intersection of all $T_1$-closed,\ $T_2$-closed sets, resp., in $X$ containing $S$. Now, if $T_1(S) = \mathrm{Int(Cl}(S))$ and $T_2(S) = \mathrm{Cl(Int}(S))$ where $S \subseteq X$, then $T_1$-open set is exactly the pre-open set and  $T_2$-open set is exactly the semi-open set. In addition, we have that $T_1\mathrm{Cl}(S) \equiv \mathrm{pCl}(S)$ and $T_2\mathrm{Cl}(S) \equiv \mathrm{sInt}(S)$.
 \end{remark}
\begin{definition}
  Let $(X, \tau)$ be a topological space and $T_1, T_2$ be two operators associated with the topology $\tau$ on $X$ that is $U \subseteq T_1(U)$ and $U \subseteq T_2(U)$ for each $U \in \tau$. The quadruple $(X, \tau, T_1, T_2)$ is called a \textsl{bi-operator topological space}.
\end{definition}
\begin{example}\item
\begin{itemize}
  \item [(i)] If $T_1, T_2$ are the identity operators, i.e. $T_1(S)= S$ and $T_2(S)= S$, then the quadruple $(X, \tau, T_1, T_2)$ will reduces to $(X, \tau)$, thus the bi-operator topological space is the ordinary
topological space.
  \item [(ii)] Let $(X, \tau)$ be any topological space and $T_1, T_2: P(X) \to P(X)$ be functions such that $T_1(S) := \mathrm{Int(Cl}(S))$ and $T_2(S) := \mathrm{Cl(Int}(S))$ for any $S \subseteq X$.
  Notice that if $U$ is open in $X$, then $U \subseteq \mathrm{Int(Cl}(U)) = T_1(U)\ \text{and} \ U \subseteq \mathrm{Cl( Int}(U)) = T_2(U).$ Thus, $T_1, T_2$ are operators associated with the topology $\tau$ on $X$ and the quadruple $(X, \tau, T_1, T_2)$ is a bi-operator topological space.
\end{itemize}

\end{example}

\begin{definition}
Let $(X, \tau, T_1, T_2)$ be a bi-operator topological space and $S \subseteq X$. The set $S$ is said to be a \textsl{$\mathcal{B}$-open} set if $$S \subseteq T_1(S) \cup T_2(S).$$
\end{definition}
The complement of a $\mathcal{B}$-open set is \textsl{$\mathcal{B}$-closed}. Moreover, if $T_1(S)= \mathrm{Cl(Int}(S))$ and $T_2(S)= \mathrm{Int(Cl}(S))$, then $S$ is $\mathcal{B}$-open if and only if $S$ is $b$-open, so the concepts of $\mathcal{B}$-openness reduces to the concepts of $b$-openness in this case. Cf.  Definition \ref{11}.

\begin{remark}\label{B22}\item
\begin{itemize}
\item [(i)] As an example of $\mathcal{B}$-open set, one can consider a bi-operator topological space $(\mathbb{R},\tau_u,\\T_1,T_2)$ such that $\mathbb{R}$ stands for the set of real numbers and $\tau_u$ for the usual topology. Let  $S\subseteq \mathbb{R}$ and $T_1(S)= \mathrm{Int(Cl}(S))$ and $T_2(S) = \mathrm{Cl(Int}(S))$. If $S = [0, 1] \cup \big((1,2) \cap Q\big),$ $Q$ denotes the set of the rational numbers then $S$ is $\mathcal{B}$-open but neither $T^*_1$-open nor $T^*_2$-open set. On other hand, if $E = [0, 1) \cup Q$, then $E$ is $T^*_1$-open but not $T^*_2$-open while $E$ is $\mathcal{B}$-open.
  \item [(ii)] The intersection of two $\mathcal{B}$-open sets is not necessarily $\mathcal{B}$-open. So, the collection of all $\mathcal{B}$-open sets is not necessarily a topology on X.
  \item [(iii)] The intersection of any collection of $\mathcal{B}$-closed sets is $\mathcal{B}$-closed. $\mathcal{B}\mathrm{Cl}(S)$ is the intersection of all $\mathcal{B}$-closed sets containing $S$, i.e. $\mathcal{B}\mathrm{Cl}(S) := \cap \ \{ U \mid  U \ \text{is}\  \mathcal{B}\text{-closed},\ U \supseteq S \}$.
   \item [(iv)] $\mathcal{B}\mathrm{Int}(S)$ is the union of all $\mathcal{B}$-open sets contained in $S$, i.e. $\mathcal{B}\mathrm{Int}(S) := \cup \ \{ U \mid  U \ \text{is}\  \mathcal{B}\text{-open},\ U \subseteq S \}$.
  \item [(v)] Every $T_1^*$-open ($T_2^*$-open) set is $\mathcal{B}$-open because if we assume that $S$ is $T_1^*$-open then $S \subseteq T_1(S) \subseteq T_1(S) \cup T_2(S)$, therefore, $S$ is $\mathcal{B}$-open and the same for the $T_2^*$-open. More precisely, if we put $T_{12}^*$-open instead of $\mathcal{B}$-open, then we have \\
               {$T_1$-open  $\to$ open $\to$ $T_1^*$-open $\to$ $T_{12}^*$-open $\to$ $T_{123}^*$-open $\to$ ... $T_{123...n}^*$-open.}

Similarly,

{$T_2$-open  $\to$ open $\to$ $T_2^*$-open $\to$ $T_{12}^*$-open $\to$ $T_{123}^*$-open $\to$ ... $T_{123...n}^*$-open},

means that $S$ is $T_{123}^*$-open if
$$S \subseteq T_1(S)  \cup T_2(S)\cup T_3(S),$$
and $T_{123...n}^*$-open means analogously $$S \subseteq T_1(S)  \cup T_2(S)\cup T_3(S) \cup ... \cup T_n(S).$$

\end{itemize}
\end{remark}

\begin{definition}
The graph $G(f)$ of a function from a bi-operator topological space $(X,\tau,T_1,\\T_2)$ into a topological space $(Y, \sigma)$ is said to be
\begin{itemize}
    \item [(i)]\textsl{$\mathcal{B}$-regular graph}, if for every $(x,y) \in X \times Y \setminus G(f)$, there exists $U$ which is $\mathcal{B}$-closed in $X$ containing $x$ and a regular open set $V$ in $Y$ containing $y$ such that $(U \times V) \cap G(f)= \emptyset.$
    \item [(ii)] \textsl{contra-$\mathcal{B}$-closed graph}, if for each $(x,y) \in X \times Y\setminus G(f),$ there exists a $\mathcal{B}$-closed set $U$ in $X$ containing $x$ and a regular closed set $V$ in $Y$ containing $y$ such that $f(U) \cap V = \emptyset.$
    \end{itemize}

    \end{definition}
    \begin{definition} \cite{D1}
A function $f: (X, \tau) \to (Y, \sigma)$ is said to be \textsl{contra-continuous}, if $f^{-1}(V)$ is closed in $X$ for each open subset $V$ of $Y$.
\end{definition}

\begin{definition}
  A function $f: (X, \tau, T_1, T_2) \to (Y, \sigma)$ is said to be \textsl{contra-$\mathcal{B}$-continuous}, if $f^{-1}(V)$ is $\mathcal{B}$-closed in $X$ for each open subset $V$ of $Y$.
\end{definition}

\begin{definition}
  Let $(X, \tau, T_1, T_2)$ be a bi-operator topological space, then $X$ is called a \textsl{$\mathcal{B}$-Frechet}, if for each pair of distinct points $x_1, x_2$ of $X$, there exists $\mathcal{B}$-open sets $U$ and $V$ containing $x_1$ and  $x_2$, respectively where $x_2 \notin U$  and  $x_1 \notin V.$ This is equivalent to saying that each single $\{x\}$ is $\mathcal{B}$-closed.
\end{definition}
\begin{definition}
   \textnormal{A topological space $(X, \tau)$ is said to be (see \cite{A2}, \cite{E1}, \cite{N3} \cite{W1} and \cite{SM})}:
  \begin{itemize}

    \item [(i)] \textsl{compact}, if for every open cover of $X$ has finite subcover.
 \item [(ii)] \textsl{contra-compact}, if for every closed cover of $X$ has finite subcover.
    \item [(iii)] \textsl{$R$-compact}, if for every regular open cover of $X$ has finite subcover.
    \item [(iv)] \textsl{contra-$R$-compact}, if for every regular closed cover of $X$ has finite subcover.
    \item [(v)] \textsl{$R$-Lindel\"{o}f}, if for every regular open cover of $X$ has countable subcover.
    \item [(vi)] \textsl{contra-$R$-Lindel\"{o}f}, if for every regular closed cover of $X$ has countable subcover.
    \item [(vii)] \textsl{countable-$R$-compact}, if for every countable regular open cover of $X$ has finite subcover.
    \item [(viii)] \textsl{contra countable-$R$-compact}, if for every countable regular closed cover of $X$ has finite subcover.

  \end{itemize}

\end{definition}

\begin{definition}
 We call the bi-operator topological space $(X, \tau, T_1, T_2)$:
   \begin{itemize}
    \item [(i)] \textsl{$\mathcal{B}$-compact}, if for every $\mathcal{B}$-open cover of $X$ has finite subcover.
    \item [(ii)] \textsl{$\mathcal{B}$-Lindel\"{o}f}, if for every $\mathcal{B}$-open cover of $X$ has countable subcover.
    \item [(iii)] \textsl{countable-$\mathcal{B}$-compact}, if for every countable-$\mathcal{B}$-open cover of $X$ has finite subcover.
  \end{itemize}

\end{definition}

\begin{definition}
 A subset $S$ of a bi-operator topological space $(X, \tau, T_1, T_2)$ is said to be \textsl{$\mathcal{B}$-dense}, if $\mathcal{B}\mathrm{Cl}(S) = X$.
\end{definition}
\begin{remark}
  If $T_1(S)= \mathrm{Int(Cl}(S)), \ T_2(S)= \mathrm{Cl(Int}(S))$, then $\mathcal{B}$-dense will be $b$-dense and $\mathcal{B}\mathrm{Cl}(S)$ will be $b\mathrm{Cl}(S)$ such that $b$-dense is a set in $X$ if $b\mathrm{Cl}(S)= X.$
\end{remark}

\begin{definition}
    A bi-operator topological space $(X, \tau, T_1, T_2)$ is called a \textsl{$\mathcal{B}$-connected} provided $X$ is not a union of two nonempty $\mathcal{B}$-open sets.

\end{definition}
\begin{definition}
   A topological space $(X, \tau)$ is said to be a \textsl{weakly Hausdorff space} \cite{S1}, if each element of $X$ is an intersection of regular closed sets.
\end{definition}
\begin{definition}
  A topological space $(X, \tau)$ is an \textsl{Urysohn space} \cite{A6}, if for every pair of distinct points $x$ and $y$ in $X$, there exist open sets $U$ and $V$ such that $x \in U, \ y \in V$ and $\mathrm{Cl}(U) \cap \mathrm{Cl}(V) = \emptyset$.
\end{definition}
\section{Some properties of $\mathcal{B}$-open sets}

\begin{lemma}
 Let $(X, \tau, T_1, T_2)$  be a bi-operator topological space given by $$T_1(S)= \mathrm{Int(Cl}(S)), T_2(S)= \mathrm{Cl(Int}(S)).$$ Then
  \begin{itemize}
  \item [(i)] $\mathcal{B}\mathrm{Int}(S) = \mathrm{sInt}(S) \cup \mathrm{pInt}(S).$
    \item [(ii)] $\mathcal{B}\mathrm{Cl}(S) = \mathrm{sCl}(S) \cap \mathrm{pCl}(S).$

  \end{itemize}
\end{lemma}
 \begin{proof}
     It is sufficient to prove only the first assertion. As we have stated in Remark \ref{B22} $(iv)$ that $\mathcal{B}\mathrm{Int}(S)$ is the union of all $\mathcal{B}$-open sets contained in $S$, therefore
$$
  \mathcal{B}\mathrm{Int}(S) \supset \mathrm{Cl(Int}(\mathcal{B}\mathrm{Int}(S))) \cup \mathrm{Int(Cl}(\mathcal{B}\mathrm{Int}(S)))
   \supset \mathrm{Cl(Int}(S)) \cup \mathrm{Int(Cl}(S))).$$
 Thus, with the help of Proposition \ref{1}, we obtain
 \begin{align*}
    \mathcal{B}\mathrm{Int}(S)  & = S \cap [ \mathrm{Cl(Int}(S)) \cup \mathrm{Int(Cl}(S))] \\
     & = [S \cap  \mathrm{Cl(Int}(S))]  \cup [S \cap  \mathrm{Int(Cl}(S))]\\
     & = \mathrm{sInt}(S) \cup \mathrm{pInt}(S).
 \end{align*}
The opposite direction is evident. One can prove the second statement in a similar way.
  \end{proof}

\begin{lemma}
    Let $(X, \tau, T_1, T_2)$ be a bi-operator topological space, suppose that $$T_1(W \cap Z)= T_1(W) \cap T_1(Z)$$ and $$T_2(W \cap Z)= T_2(W) \cap T_2(Z),$$ for all $W \in \tau, Z \subseteq X$ then the following assertions are satisfied:
  \begin{itemize}
    \item [(i)] The intersection of an open set with a $\mathcal{B}$-open set is a $\mathcal{B}$-open set.
    \item [(ii)] The union of any family of $\mathcal{B}$-open sets is a $\mathcal{B}$-open set.
  \end{itemize}

\end{lemma}
\begin{proof}
  (i) Assume that there exists $U \in \tau$, which is an open set, and $V$ is a $\mathcal{B}$-open set. We are going to show that $U \cap V $ is also a $\mathcal{B}$-open set. Since $U$ is open, then $$U \subseteq  T_1(U), \ U \subseteq  T_2(U.$$ By the definition of the $\mathcal{B}$-open set: $$ V \subseteq  T_1(V) \cup  T_2(V).$$
  Now,
  \begin{align*}
    U \cap V & \subseteq U \cap [ T_1(V) \cup  T_2(V)]\\
    & = [U \cap  T_1(V)] \cup [U \cap  T_2(V)] \\
     & \subseteq [T_1(U) \cap  T_1(V)] \cup [T_2(U) \cap  T_2(V)] \\
    & = [T_1(U \cap  V)] \cup [T_2(U \cap  V)],
  \end{align*}
as wanted to be shown.\\
  (ii) Suppose that $\mathcal{F}= \{V_a| \ a \in \lambda \}$ is a family of $\mathcal{B}$-open set, $$V_a \subseteq  T_1(V_a) \cup  T_2(V_a).$$ Then we have,
  \begin{align*}
    \bigcup_a V_a & \subseteq \bigcup_a ( T_1(V_a) \cup  T_2(V_a)) \\
     & = \bigcup_a T_1(V_a) \cup \bigcup_a T_2(V_a).
  \end{align*}
  It is clear that $\bigcup_a T_1(V_a)= T_1(\bigcup_a V_a)$ and $\bigcup_a T_2(V_a)= T_2(\bigcup_a V_a),$ therefore $$ \bigcup_a V_a  \subseteq T_1(\bigcup_a V_a) \cup T_2(\bigcup_a V_a).$$ Thus, $ \bigcup_a V_a$ is a $\mathcal{B}$-open set, which completes the proof.
\end{proof}

\begin{proposition}
   \textnormal{Let $(X, \tau, T_1, T_2)$ be a bi-operator topological space. If the function $f: (X, \tau, T_1, T_2) \to (Y, \sigma)$ has a contra-$\mathcal{B}$-closed graph, then the inverse image of a contra-compact set $S$ of $Y$ is $\mathcal{B}$-closed in $X$.}
\end{proposition}

\begin{proof}
  Assume that $S$ is a contra-compact set of $Y$ and $x \notin f^{-1}(S)$, i.e. for all $a \in S, (x, a) \notin G(f)$. Then there exist $U_a$ which is $\mathcal{B}$-closed containing $x$ and $V_a$ which is closed in $Y$ containing $a$ such that $$f(U_a) \cap V_a = \emptyset.$$
  On the other hand one can consider $\mathcal{F}=\{S \cap V_a | \ a \in S\}$ and $\mathcal{F}$ is closed cover of the subspace $S$. We have that $S$ is contra-compact, then there exists $a_1, a_2,..., a_n$ such that $S \subseteq \cup_{i=1}^n V_{a_i}$. Now, if $U = \cap_{i=1}^n U_{a_i}$, then $U$ is $\mathcal{B}$-closed containing $x$ and  $f(U) \cap S = \emptyset$, therefore $U \cap f^{-1}(S) = \emptyset.$ Hence $x \notin \mathcal{B}\mathrm{Cl}(f^{-1}(S))$, this shows that $f^{-1}(S)$ is $\mathcal{B}$-closed.
\end{proof}

\begin{proposition}
   Let $f:(X, \tau, T_1, T_2) \to (Y, \sigma)$ from a bi-operator topological space to a contra-compact space which has a contra-$\mathcal{B}$-closed graph, then $f$ is a contra-$\mathcal{B}$-continuous function.
\end{proposition}
\begin{proof}
  Let $\mathcal{F}= \{V_a | \ a \in \lambda \}$ be a cover of an open set $U \subset Y$ by the closed subsets $V_a$ of $U$ for each $a \in \lambda$. Thus, there exists a closed set $W_{a}$ of $Y$ where $V_a= W_{a} \cap \ U,$ i.e. $\{W_a | \ a \in \lambda \} \cup \{U^c\}$ is a closed cover of $Y$. But $Y$ is a contra-compact space, namely, there exist $a_1, a_2,..., a_n$ such that $Y= \cup_{i=1}^{n} W_{a_i} \cup U^c$. Hence $U= \cup_{i=1}^{n} V_{a_i},$ and consequently $U$ is contra-compact. From previous proposition $f^{-1}(U)$ is $\mathcal{B}$-closed in $X$, thus $f$ is contra-$\mathcal{B}$-continuous.
\end{proof}

The proof of the next lemma is immediate, since $g$ is contra-$\mathcal{B}$-continuous, so $f^{-1}(U) = g^{-1}(X \times U)$ is $\mathcal{B}$-closed in $X$, then $f$ is contra-$\mathcal{B}$-continuous.

\begin{lemma}
  Let $f:(X, \tau, T_1, T_2) \to (Y, \sigma)$ be a function and \ $g: (X, \tau, T_1, T_2)$ $\to (X \times Y)$ be a graph function of $f$ defined by $g(x)=(x, f(x))$ for every $x \in X$. If $g$ is contra-$\mathcal{B}$-continuous then $f$ is contra-$\mathcal{B}$-continuous.
\end{lemma}

\begin{proposition}
  Let $f : (X, \tau, T_1, T_2) \to (Y, \sigma)$ be contra-$\mathcal{B}$-continuous and $g :(X, \tau) \to (Y, \sigma)$ is contra-continuous. If $Y$ is an Urysohn space, then $E=\{ x \in X | \ f(x)=g(x)\}$ is $\mathcal{B}$-closed in $X$.
\end{proposition}
\begin{proof}
  Suppose that $x \in E^c$, this implies that $f(x) \neq g(x)$. Since $Y$ is an Urysohn space, then there exist open sets $U$ and $V$ such that $f(x) \in U, \ g(x) \in V$ and $\mathrm{Cl}(U) \cap \mathrm{Cl}(V)=\emptyset.$ Since the function $f$ is  contra-$\mathcal{B}$-continuous, $f^{-1}(\mathrm{Cl}(U))$ is $\mathcal{B}$-open in $X$ and $g$ is contra-continuous, therefore $g^{-1}(\mathrm{Cl}(V))$ is open in $X$.
  If we consider $W= f^{-1}(\mathrm{Cl}(U)), Z= g^{-1}(\mathrm{Cl}(V)),$ then $x \in W \cap Z= S$ where $S$ is $\mathcal{B}$-open in $X$ and $f(S) \cap g(S) \subseteq f(W) \cap g(Z) \subseteq \mathrm{Cl}(U) \cap \mathrm{Cl}(V) = \emptyset.$
  Hence $f(S) \cap g(S)= \emptyset$ and $S \cap E = \emptyset, S \subseteq E^c$ where $S$ is $\mathcal{B}$-open. We conclude that $x \notin \mathcal{B}\mathrm{Cl}(E)$, and so $E$ is $\mathcal{B}$-closed in $X.$
\end{proof}
\begin{corollary}
  Let $f : (X, \tau, T_1, T_2) \to (Y, \sigma)$  be contra-$\mathcal{B}$-continuous and let $g: (X, \tau) \to (Y, \sigma)$ be contra-continuous.
  If $Y$ is an Urysohn space and $ f=g$ on a $\mathcal{B}$-dense set $S \subseteq X$, then $ f=g$ on $X$.
\end{corollary}
\begin{proof}
  From the previous result $E=\{ x \in X | f(x)=g(x)\}$ is $\mathcal{B}$-closed in $X.$ Now we assumed that $f=g$ on $\mathcal{B}$-dense set and $S \subseteq E.$ Since $f$ is contra-$\mathcal{B}$-continuous and $g$ is contra-continuous,
  then $X = \mathcal{B}\mathrm{Cl}(S)\subseteq \mathcal{B}\mathrm{Cl}(E)= S$. Therefore, $f=g$ on $X$.
\end{proof}
\begin{proposition}
   If $f : (X, \tau, T_1, T_2) \to (Y, \sigma)$ is contra-$\mathcal{B}$-continuous from a $\mathcal{B}$-connected space onto $Y$, then $Y$ is not a discrete space.
\end{proposition}
\begin{proof}
  Let $Y$ be a discrete space and $\emptyset \neq S \subset Y,$ then $S$ is a proper nonempty open and closed subset of $Y$. Then $f^{-1}(S)$ is a proper nonempty $\mathcal{B}$-open and $\mathcal{B}$-closed subset of $X$ such that $X = f^{-1}(S) \cup (f^{-1}(S))^c$ which means that $X$ is $\mathcal{B}$-disconnected space and this contradicts our assumption. Thus, $Y$ is not discrete.
\end{proof}
\begin{definition}
    \textnormal{A function $f : (X, \tau, T_1, T_2) \to (Y, \sigma)$  is called an \textsl{almost contra-$\mathcal{B}$-continuous function}, if $f^{-1}(V)$ is a $\mathcal{B}$-closed for every regular open set $V$ in $Y.$}
\end{definition}
\begin{proposition}
Let  $f : (X, \tau, T_1, T_2) \to (Y, \sigma)$ be a surjective almost contra-$\mathcal{B}$-continuous function, then:
  \begin{itemize}
    \item [(i)] if $X$ is $\mathcal{B}$-Lindel\"{o}f, then $Y$ is contra-$R$-Lindel\"{o}f.
    \item [(ii)] if $X$ is $\mathcal{B}$-compact, then $Y$ is contra-$R$-compact.
    \item [(iii)] if $X$ is countable-$\mathcal{B}$-compact, then $Y$ is countable contra-$R$-compact.
  \end{itemize}

\end{proposition}
\begin{proof}
  We are going to prove (i) and (ii) and one can prove (iii) in a similar way.\\
  (i) Consider a family $\mathcal{S}=\{ V_a | \ a \in \lambda\}$ to be a regular closed cover of $Y,$ at the same time let $\mathcal{S^*}=\{ f^{-1}(V_a) | \ a \in \lambda\}$ be a $\mathcal{B}$-open cover of $X.$ But $X$ is $\mathcal{B}$-Lindel\"{o}f, then there exist $a_1, a_2, ..., a_n$ such that $X=\cup_{i=1}^{\infty} f^{-1}(V_{a_i})$, we have $Y = f(X) = f(\cup_{i=1}^{\infty}f^{-1}(V_{a_i}))$. Then $Y = \cup_{i=1}^{\infty}(V_{a_i})$ is contra-$R$-Lindel\"{o}f.\\
  (ii) Using the same technique as above, let $\mathcal{F}=\{ U_a | \ a \in \lambda\}$ be a regular closed cover of $Y$ since $f$ is a surjective almost contra-$\mathcal{B}$-continuous function. So $\mathcal{F^*}=\{ f^{-1}(U_a) | \ a \in \lambda\}$ is a $\mathcal{B}$-open cover of $X$ but $X$ is $\mathcal{B}$-compact, then there exists $a_1, a_2, ..., a_n$ where $X=\cup_{i=1}^n f^{-1}(U_{a_i}).$ Consequently, $$Y = f(X) = f(\cup_{i=1}^nf^{-1}(U_{a_i}))= \cup_{i=1}^n(U_{a_i}).$$ This clearly forces $Y$ to be contra-$R$-compact.

\end{proof}
\begin{definition} \cite{N31}
 A function $f : X \to Y$  is called:
\begin{itemize}
   \item \textsl{almost continuous}, if $f^{-1}(V)$ is open in $X$ for every regular open set $V$ in $Y$.
  \item \textsl{$R$-continuous}, if $f^{-1}(V)$ is a regular open set of $X$ for each regular closed set $V$ in $Y$.
\end{itemize}

\end{definition}
\begin{lemma} \cite{N31}
 If a function $f : X \to Y$  is almost contra-$b$-continuous and almost continuous, then $f$ is a $R$-continuous function.
\end{lemma}
\begin{proposition}
Let  $f : (X, \tau, T_1, T_2) \to (Y, \sigma)$ be an almost contra-$\mathcal{B}$-continuous and surjective almost-continuous function,
suppose $T_1(S)= \mathrm{Int(Cl}(S))$ and $T_2(S)= \mathrm{Cl(Int}(S))$, then $Y$ is:
  \begin{itemize}
    \item [(i)] contra-$R$-compact, if $X$ is contra-$R$-compact.
    \item [(ii)] $R$-compact, if $X$ is $R$-compact.
    \item [(iii)] $R$-Lindel\"{o}f, if $X$ is $R$-Lindel\"{o}f.
    \item [(iv)] countable-$R$-compact, if $X$ is countable$R$-compact.
    \item [(v)] countable contra-$R$-compact, if $X$ is countable contra-$R$-compact.
    \item [(vi)] contra-$R$-Lindel\"{o}f, if $X$ is contra-$R$-Lindel\"{o}f.
  \end{itemize}

\end{proposition}
\begin{proof}
It is enough to prove (i) and for the rest one can use the same methods to prove them.\\
(i) $T_1(S)= \mathrm{Int(Cl}(S))$ and $T_2(S)= \mathrm{Cl(Int}(S))$ are given. So $f$ is almost contra-$\mathcal{B}$-continuous and surjective almost-continuous, by the above lemma, $f$ is $R$-continuous, that is the inverse of each regular closed set in $Y$ is regular in $X$. Assume that $\mathcal{S}=\{ V_a | \ a \in \lambda\}$ is a regular closed cover of $Y.$ Consequently, $\mathcal{S^*}=\{ f^{-1}(V_a) | \ a \in \lambda\}$ is a regular closed cover of $X$, but $X$ is contra-$R$-compact, therefore there exists $a_1, a_2, ..., a_n$ such that $X=\cup_{i=1}^n f^{-1}(V_{a_i})$ and $Y = f(X) = f(\cup_{i=1}^nf^{-1}(V_{a_i}))$, which shows that $Y$ is contra-$R$-compact.
\end{proof}
\begin{proposition}
  If $f : (X, \tau, T_1, T_2) \to (Y, \sigma)$ is a contra-$\mathcal{B}$-continuous function and $S$ is $\mathcal{B}$-compact relative to $X$, then $f(X)$ is contra-compact in $Y.$
\end{proposition}
\begin{proof}
  Let $\mathcal{F}= \{V_a | \ a \in \lambda \}$ be any cover of $f(S)$. It follows from the closed set of the subspace of $f(S)$ for all $a\in \lambda$ that there exists a closed set $S_a$ of $Y$ such that $S_a \cap f(S)= V_a$ and for each $x \in S$, there exists $a(x) \in \lambda$ where $f(S) \in S_{a(x)}$. Then there exists $U_x$ which is $\mathcal{B}$-open, this implies that $f(U_x) \in S_{a(x)}$ such that the family $\mathcal{F^*}= \{U_x | \ x \in S \}$ is a cover of $S$ by $\mathcal{B}$-open of $X$. But $S$ is $\mathcal{B}$-compact relative to $X$, so there exist $x_1,...,x_n \in S$ and $S \subseteq \cup_{i=1}^n f^{-1} U_{x_i}.$ Hence $f(S) \subseteq f(\cup_{i=1}^n f^{-1} U_{x_i})= \cup_{i=1}^n U_{x_i}$, therefore, $f(S) = \cup_{i=1}^n V_{a(x_i)}$.
\end{proof}
\begin{corollary}
  If $f : (X, \tau, T_1, T_2) \to (Y, \sigma)$ is a contra-$\mathcal{B}$-continuous surjective function and $X$ is $\mathcal{B}$-compact, then $Y$ is contra-compact.
\end{corollary}
\begin{definition}
 A function  $f : (X, \tau, T_1, T_2) \to (Y, \sigma)$ is called \textsl{almost weakly-$\mathcal{B}$-continuous}, if for each $x \in X$ and regular set $V$ containing $f(x)$ there exist $U$ which is a $\mathcal{B}$-open set in $X$ containing $x$ such that $f(U) \subseteq \mathrm{Cl}(V)$.
\end{definition}
\begin{proposition}
  Let a function  $f : (X, \tau, T_1, T_2) \to (Y, \sigma)$ be an almost contra-$\mathcal{B}$-continuous and $Y$ be an Urysohn space, then $G(f)$ is regular in $X \times Y.$
\end{proposition}
\begin{proof}
  Let $(x, y) \in X \times Y \setminus G(f)$, it follows that $y \neq f(x)$, since $Y$ is an Urysohn, then there exist open sets $U$ and $V$ containing $f(x)$ and $y$, respectively where $U \cap V = \emptyset$. Then $$\mathrm{Int(Cl}(U)) \cap \mathrm{Cl(Int}(V)) = \emptyset.$$ Since $f$ is almost contra-$\mathcal{B}$-continuous, we have that $f^{-1}(\mathrm{Int(Cl}(U)))$ is $\mathcal{B}$-closed in $X$ containing $x$. If $W= f^{-1}(\mathrm{Int(Cl}(U)))$, then $f(W) \subseteq \mathrm{Int(Cl}(U))$ such that $f(W) \cap \mathrm{Int(Cl}(V))= \emptyset$ and $\mathrm{Int(Cl}(V)$ is regular in $Y$. Hence $G(f)$ is $\mathcal{B}$-regular in $X\times Y.$
\end{proof}
\begin{proposition}
  Suppose that $f : (X, \tau, T_1, T_2) \to (Y, \sigma)$ has a $\mathcal{B}$-regular graph. If $f$ is a surjective function, then $Y$ is weakly Hausdorff.
\end{proposition}
\begin{proof}
  Let $y,\bar{y}$ be any two distinct points of $Y$. Since $f$ is surjective, then there exists $x \in X$ where $f(x)=y$. Notice $(x, \bar{y}) \in (X \times Y)\setminus G(f)$, by the definition of $\mathcal{B}$-regular graph, there exists $\mathcal{B}$-closed set $U$ of $X$ and a regular open $F_{\bar{y}}$ in $Y$ such that $(x, \bar{y}) \in U \times F_{\bar{y}}$ and $f(U) \cap F_{\bar{y}}=\emptyset.$ Since $f(x) \in f(U)$ and $y \notin F_{\bar{y}},$ $\bar{y} \notin F_{\bar{y}}^c$ which is regular closed in $Y,$ we get $y = \underset{\bar{y} \neq y} \cap F_{\bar{y}}^c$. Thus, $Y$ is weakly Hausdorff.
\end{proof}
\begin{proposition}
  Let $f : (X, \tau, T_1, T_2) \to (Y, \sigma)$ be a function which has a $\mathcal{B}$-regular graph. If $f$ is an injective function, then $X$ is $\mathcal{B}$-Frechet.
\end{proposition}
\begin{proof}
  Assume that $x_1,x_2$ are any two distinct points of $X$. Since $f$ is injective, it follows that $(x_1, f(x_2)) \in (X \times Y)\setminus G(f)$, by the definition of $\mathcal{B}$-regular graph. Then there exist a $\mathcal{B}$-closed set $U$ of $X$ and a regular open set $V$ in $Y$ such that $(x_1, f(x_2)) \subset U \times V$ and $f(U) \cap V = \emptyset,$ therefore $U \cap f^{-1}(V)=\emptyset$ and $x_2 \notin U$. Thus $x_1 \notin U^c, \ x_2 \in U^c$ and $U^c$ is $\mathcal{B}$-open which means that $\{x_1\}^c$ is $\mathcal{B}$-open, that is  $\{x_1\}$ is $\mathcal{B}$-closed. So $X$ is is $\mathcal{B}$-Frechet.
\end{proof}
\begin{proposition}
   Let $f : (X, \tau, T_1, T_2) \to (Y, \sigma)$ be a weakly-$\mathcal{B}$-continuous function and $Y$ be an Urysohn space. Then $G(f)$ is contra $\mathcal{B}$-regular in $X \times Y.$
\end{proposition}
\begin{proof}
  Let us consider $(x, y) \in (X \times Y) \setminus G(f)$, therefore $y \neq f(x)$. Since $Y$ is an Urysohn space, then there exist two open sets $U$ and $V$ in $Y$ containing $y$ and $f(x)$, respectively. Consider that $W$ is a $\mathcal{B}$-open set containing $x$ and $\mathrm{Cl}(U) \cap \mathrm{Cl}(W)=\emptyset.$ Since we are working under the assumption that
    $f$ is weakly-$\mathcal{B}$-continuous, then $f(x) \subseteq \mathrm{Cl}(V)$ which implies that $ f(W) \cap \mathrm{Cl}(U) =  f(W) \cap \mathrm{Cl(Int}(U))= \emptyset$ and $\mathrm{Cl(Int}(U))$ is regular closed containing $y.$ Hence $G(f)$ is a contra $\mathcal{B}$-regular graph in $X \times Y,$ which completes the proof.

\end{proof}

\paragraph{{\bf}ACKNOWLEDGEMENTS} I would like to express my sincere gratitude to my supervisor Dr. L\'{a}szl\'{o} Kozma, for carefully reviewing this paper, providing beneficial suggestions.

\end{document}